\documentclass[12pt]{amsart}

\usepackage{amscd,latexsym,amsthm,amsfonts,amssymb,amsmath,amsxtra}
\usepackage[mathscr]{eucal}
\usepackage{pictexwd,dcpic}
\usepackage[normalem]{ulem}
\usepackage{hyperref}
\renewcommand\theequation{\thesection.\arabic{equation}}

\newcommand{\BC}{{\mathbb {C}}}

\newcommand{\BR}{{\mathbb {R}}}

\newcommand{\BZ}{{\mathbb {Z}}}

\newcommand{\CL}{{\mathcal {L}}}

\newcommand{\CP}{{\mathcal {P}}}

\newcommand{\Fa}{{\mathfrak {a}}}

\newcommand{\Fg}{{\mathfrak {g}}}

\newcommand{\Fl}{{\mathfrak {l}}}

\newcommand{\GL}{{\mathrm{GL}}}

\newcommand{\tr}{{\mathrm{tr}}}

\newcommand{\zh}{Z_H(F)\backslash H(F)}

\newtheorem{thm}{Theorem}[section]

\newtheorem{prop}[thm]{Proposition}
\newtheorem {conj}[thm]{Conjecture}

\newtheorem {ques/conj}[thm]{Question/Conjecture}

\newtheorem{rmk}[thm]{Remark}

\makeatletter

\newcommand{\Rmnum}[1]{\expandafter\@slowromancap\romannumeral #1@}
\makeatother

\begin{document}
\renewcommand{\theequation}{\arabic{equation}}
\numberwithin{equation}{section}

\title[On the Ginzburg-Rallis models]{The Local Ginzburg-Rallis Model Over Complex Field.}

\author{Chen Wan}
\address{School of Mathematics\\
University of Minnesota\\
Minneapolis, MN 55455, USA}
\email{wanxx123@umn.edu}

\subjclass[2010]{Primary 22E50}
\date{\today}
\keywords{Representations of Linear Algebraic Groups over archimedean local field, Multiplicity One on Vogan Packet}

\begin{abstract}
We consider the local Ginzburg-Rallis model over complex field. We show that the multiplicity is always 1 for a majority of the generic representations. We also have partial results on the real case for general generic representations. This is a sequel work of \cite{Wan15} and \cite{Wan16} on which we considered the p-adic case and the real case for tempered representations.
\end{abstract}

\maketitle

\tableofcontents

\section{Introduction and Main Result}
This paper is a continuation of \cite{Wan15} and \cite{Wan16}. For an overview of the Ginzburg-Rallis model, see Section 1 of \cite{Wan15}. We recall from there the definition of the Ginzburg-Rallis models and the conjectures.

Let $F$ be a local field (p-adic or archimedean), $D$ be the unique quaternion algebra over $F$ if $F\neq \BC$. Take $P=P_{2,2,2}=MU$ be the standard parabolic subgroup of $G=\GL_6(F)$ whose Levi part $M$ is isomorphic to $\GL_2\times \GL_2\times \GL_2$, and whose unipotent radical $U$ consists of elements of the form
\begin{equation}\label{unipotent}
u=u(X,Y,Z):=\begin{pmatrix} I_2 & X & Z \\ 0 & I_2 & Y \\ 0 & 0 & I_2 \end{pmatrix}.
\end{equation}
We define a character $\xi$ on $U$ by
\begin{equation}\label{character}
\xi(u(X,Y,Z)):=\psi(\tr(X)+\tr(Y))
\end{equation}
where $\psi$ is a non-trivial additive character on $F$. It's clear that the stabilizer of $\xi$ is the diagonal embedding of $\GL_2$ into $M$, which is denoted by $H_0$. For a given character $\chi$ of $F^{\times}$, one induces a one dimensional representation $\omega$ of $H_0(F)$ given by $\omega(h):=\chi(\det(h))$. We can extend the character $\xi$ to the semi-direct product
\begin{equation}
H:=H_0\ltimes U
\end{equation}
by making it trivial on $H_0$. Similarly we can extend the character $\omega$ to $H$. It follows that the one dimensional representation $\omega\otimes \xi$ of $H(F)$ is well defined. The pair $(G,H)$ is the Ginzburg-Rallis model introduced by D. Ginzburg and S. Rallis in their paper \cite{GR00}. Let $\pi$ be an irreducible admissible representation of $G(F)$ with central character $\chi^2$, we are interested in the Hom space
$Hom_{H(F)}(\pi,\omega\otimes \xi)$, the dimension of which is denoted by $m(\pi)$ and is called be the multiplicity.

On the other hand, if $F\neq \BC$, define $G_D=\GL_3(D)$, similarly we can define $U_D,\;H_{0,D}$ and $H_D$. We also define the character $\omega_D\otimes \xi_D$ on $H_D(F)$ in the same way except that the trace in the definition of $\xi$ is replaced by the reduced trace of the quaternion algebra $D$ and the determinant in the definition of $\omega$ is replaced by the reduced norm of the quaternion algebra $D$. Then for an irreducible admissible representation $\pi_D$ of $G_D(F)$ with central character $\chi^2$, we can also talk about the Hom space
$Hom_{H_D(F)}(\pi_D,\omega_D\otimes \xi_D)$, whose dimension is denoted by $m(\pi_D)$.

The purpose of this paper is to study the multiplicity $m(\pi)$ and $m(\pi_D)$. First, it was proved by C.-F. Nien in \cite{N06} over a $p$-adic local field, and by D. Jiang, B. Sun and C. Zhu in \cite{JSZ11} for an archimedean local field that
both multiplicities are less or equal to 1: $m(\pi),\;m(\pi_D)\leq 1.$
In other word, the pairs $(G,H)$ and $(G_D,H_D)$ are Gelfand pairs. In this paper,
we are interested in the relation between $m(\pi)$ and $m(\pi_D)$ under the local Jacquet-Langlands correspondence established in \cite{DKV84}. The local conjecture has been expected since the work of \cite{GR00}, and was first discussed in details by Jiang in his paper \cite{J08}.

\begin{conj}[Jiang,\cite{J08}]\label{jiang}
For any irreducible admissible representation $\pi$ of $GL_6(F)$, let $\pi_D$ be the local Jacquet-Langlands correspondence of $\pi$ to $GL_3(D)$ if it exists, and zero otherwise. In particular, $\pi_D$ is always 0 if $F=\BC$. We still assume that the central character of $\pi$ is $\chi^2$. Then the following identity
\begin{equation}\label{equation 1}
m(\pi)+m(\pi_D)=1
\end{equation}
holds for all irreducible generic representation $\pi$ of $GL_6(F)$.
\end{conj}

Note that the assertion in Conjecture \ref{jiang} can be formulated in terms of Vogan packets and pure inner forms of $PGL_6$. We refer to
\cite{Wan15} for discussion.

Another aspect of the local conjecture is to relate the multiplicity to the central value of the exterior cube epsilon factors. It can be stated as follows.
\begin{conj}\label{epsilon}
With the same assumptions as in Conjecture \ref{jiang}, assume that the central character of $\pi$ is trivial, then we have
\begin{eqnarray*}
m(\pi)=1 &\iff& \epsilon(1/2,\pi,\wedge^3)=1,\\
m(\pi)=0 &\iff& \epsilon(1/2,\pi,\wedge^3)=-1.
\end{eqnarray*}
\end{conj}

In this paper, we always fix a Haar measure $dx$ on $F$ and an additive character $\psi$ such that the Haar measure is selfdual for Fourier transform with respect to $\psi$. We use such $dx$ and $\psi$ in the definition of the $\epsilon$ factor. For simplicity, we will write the epsilon factor as $\epsilon(s,\pi,\rho)$ instead of $\epsilon(s,\pi,\rho,dx,\psi)$.

\begin{rmk}
In Conjecture \ref{epsilon}, we do need the assumption that the central character of $\pi$ is trivial. Otherwise, the exterior cube of the Langlands parameter of $\pi$ will no longer be selfdual, and hence the value of the epsilon factor at $1/2$ may not be $\pm 1$.
\end{rmk}

In the previous papers \cite{Wan15} and \cite{Wan16}, we prove Conjecture \ref{jiang} for the case that $F$ is a $p$-adic local field or $\BR$ and $\pi$ is an irreducible tempered representation of $GL_6(F)$. In \cite{Wan16}, we also prove Conjecture \ref{epsilon} for the case $F=\BR$ and $\pi$ is tempered, together with the case when $F$ is p-adic and $\pi$ is tempered but not discrete series or parabolic induction of discrete series of $\GL_4(F)\times \GL_2(F)$.

In this paper, we considered the case when $F=\BC$. In this case, by the Langlands classification, any generic representation $\pi$ is a principal series. In other word, let $B=M_0 U_0$ be the Borel subgroup consists of all the lower triangular matrix, here $M_0=(\GL_1)^6$ is just the diagonal matrix. Then $\pi$ is of the form $I_{B}^{G}(\chi)$ where $\chi=\Pi_{i=1}^{6} \chi_i$ is a character on $M_0$ and $I_{B}^{G}$ is the normalized parabolic induction. For $1\leq i\leq 6$, we can find an unitary character $\sigma_i$ and some real number $s_i\in \BR$ such that $\chi_i=\sigma_i |\;|^{s_i}$. Without loss of generality, we assume that $s_i\leq s_j$ for any $i\geq j$. Then if we combine those representations with the same exponents $s_i$, we can find a parabolic subgroup $Q=LU_Q$ containing $B$ with $L=\Pi_{i=1}^{k}\GL_{n_i}$, a representation $\tau=\Pi_{i=1}^{k}\tau_i |\;|^{t_i}$ of $L(F)$ where $\tau_i$ are all tempered and the exponents $t_i$ are strictly decreasing (i.e. $t_1<t_2<\cdots<t_k$) such that $\pi=I_{Q}^{G}(\tau)$. On the other hand, we can also write $\pi$ as $I_{\bar{P}}^{G}(\pi_0)$ with $\pi=\pi_1\times \pi_2\times \pi_3$ and $\pi_i$ be the parabolic induction of $\chi_{2i-1}\times \chi_{2i}$.

\begin{thm}\label{main}
Assume that $F=\BC$, with the same assumptions as in Conjecture \ref{jiang} and with the notation above, the following hold.
\begin{enumerate}
\item If $\bar{P}\subset Q$, Conjecture \ref{jiang} and Conjecture \ref{epsilon} hold. In particular, both conjectures hold for the tempered representations.
\item If $Q\subset \bar{P}$ and if $\pi_0$ satisfies the condition (40) in \cite{L01}, Conjecture \ref{jiang} and Conjecture \ref{epsilon} hold.
\end{enumerate}
\end{thm}

There are two main ingredients in our proof. First we deal with the tempered representations. The idea is to construct an explicit element inside the Hom space given by integrating the matrix coefficient. Then we show that the nonvanishing property of this element is invariant under parabolic induction which allows us to reduce to the torus case which is trivial. This idea already appears in \cite{Wan16} for the case when $F=\BR$.

Then for general generic representations, we use the open orbit method to reduce our problems to the tempered case or the trilinear $\GL_2$ model case. To be specific, if $\bar{P}\subset Q$, by applying the open orbit method, we can reduce to the model related to the Levi subgroup $L$. Then after twisting $\tau$ by some characters, we only need to deal with the tempered case which has already been proved in the first place. If $Q\subset \bar{P}$, by applying the open orbit method, we reduced ourselves to the trilinear $\GL_2$ model case. Then by applying the work of Loke in \cite{L01}, we can prove our result. The extra condition in part (2) of Theorem \ref{main} also comes from \cite{L01}.

It is worth to mention that in Theorem \ref{main}(2), the requirements we made for the parabolic subgroup $Q$ force some types of generalized Jacquet integrals to be absolutely convergent, this allows us to apply the open orbit method. If one can prove such integrals have holomorphic continuation, we can actually remove this restrain. This will be discussed in Section 7.

Finally, the open orbit method we used here can also be applied to the case when $F=\BR$, this will gives us partial results about Conjecture \ref{jiang} and Conjecture \ref{epsilon} for general generic representations. To be specific, let $\pi$ be a irreducible generic representation of $G(F)$ with central character $\chi^2$. By the Langlands classification, there is a parabolic subgroup $Q=LU_Q$ containing the lower Borel subgroup and an essential tempered representation $\tau=\Pi_{i=1}^{k} \tau_i |\;|^{s_i}$ of $L(F)$ with $\tau_i$ tempered, $s_i\in \BR$ and $s_1<s_2<\cdots<s_k$ such that $\pi=I_{Q}^{G}(\tau)$. We say $Q$ is nice if $Q\subset \bar{P}$ or $\bar{P}\subset Q$.

\begin{thm}\label{F=R}
With the notations above, the following hold.
\begin{enumerate}
\item If $\pi_D=0$, assume that $Q$ is nice, then Conjecture \ref{jiang} and Conjecture \ref{epsilon} hold.
\item If $\pi_D\neq 0$, we have
$$m(\pi)+m(\pi_D)\geq 1.$$
Moreover if the central character of $\pi$ is trivial (as in Conjecture \ref{epsilon}), we have
$$\epsilon(1/2,\pi,\wedge^3)=1\Rightarrow m(\pi)=1; \; m(\pi)=0\Rightarrow \epsilon(1/2,\pi,\wedge^3)=-1.$$
\end{enumerate}
\end{thm}
As in the complex case, the assumption on $Q$ can be removed if we can prove the holomorphic continuation of certain generalized Jacquet integrals. This will also be discussed in Section 7.

The paper is organized as following: In Section 2, we review a well know result of the intertwining operator which is due to Harish-Chandra. We will also give a brief overview of the open orbit method which will be used in later sections. In Section 3, we show that for $F=\BC$, Conjecture \ref{jiang} implies Conjecture \ref{epsilon}. In Section 4, we prove Theorem \ref{main} for the tempered representations. Then in Section 5, we prove it for the general cases. In Section 6, we discuss the case for $F=\BR$. In Section 7, we will talk about how to remove the assumptions on $Q$ based on the results on the holomorphic continuation of the generalized Jacquet integral due to Raul Gomez in \cite{G}.

\textbf{Acknowledgement}: I would like to thank my advisor Dihua Jiang for suggesting me thinking about this problem. I would like to thank Omer Offen for helpful discussions of the open orbit method. I would like to thank Nolan Wallach and Raul Gomez for answering my questions on the generalized Jacquet integrals.

\section{Preliminary}
\subsection{The intertwining operator}
For every connected reductive algebraic group $G$ defined over $F$, let $A_G$ be the maximal split center of $G$ and $Z_G$ be the center of $G$.
We denote by $X(G)$ the group of $F$-rational characters of $G$. Define $\Fa_G=$Hom$(X(G),\BR)$, and let $\Fa_{G}^{\ast}=X(G)\otimes_{\BZ} \BR$ be the dual of $\Fa_G$. We define a homomorphism $H_G:G(F)\rightarrow \Fa_G$ by $H_G(g)(\chi)=\log(|\chi(g)|_F)$ for every $g\in G(F)$ and $\chi\in X(G)$.

Given a parabolic subgroup $P=MU$ of $G$ and an admissible representation $(\tau,V_{\tau})$ of $M(F)$, let $(I_{P}^{G}(\tau),I_{P}^{G}(V_{\tau}))$ be the normalized parabolic induced representation: $I_{P}^{G}(V_{\tau})$ consist of smooth functions $e:G(F)\rightarrow V_{\tau}$ such that
$$e(mug)=\delta_P(m)^{1/2} \tau(m)e(g), \;m\in M(F),\; u\in U(F),\; g\in G(F).$$
And the $G(F)$ action is just the right translation.

For $\lambda\in \Fa_{M}^{\ast}\otimes_{\BR}\BC$, let $\tau_{\lambda}$ be the unramified twist of $\tau$, i.e. $\tau_{\lambda}(m)=\exp(\lambda(H_M(m))) \tau(m)$ and let $I_{P}^{G}(\tau_{\lambda})$ be the induced representation. By the Iwasawa decomposition, every function $e\in I_{P}^{G}(V_{\tau})$ is determined by its restriction on $K$, and that space is invariant under the unramified twist. i.e. for any $\lambda$, we can realize the representation $I_{P}^{G}(\tau)$ on the space $I_{K\cap P}^{K}(\tau_K)$ which is consisting of functions $e_K:K\rightarrow V_{\tau}$ such that
$$e(mug)=\delta_P(m)^{1/2} \tau(m)e(g), \;m\in M(F)\cap K,\; u\in U(F)\cap,\; g\in K.$$

Now we define the intertwining operator. For a Levi subgroup $M$ of $G$, $P=MU,P'=MU'\in \CP(M)$, and $\lambda\in \Fa_{M}^{\ast}\otimes_{\BR} \BC$, define the intertwining operator $J_{P'|P}(\tau_{\lambda}):I_{P}^{G}(V_{\tau})\rightarrow I_{P'}^{G}(V_{\tau})$ to be
$$J_{P'|P}(\tau_{\lambda})(e)(g)=\int_{(U(F)\cap U'(F))\backslash U'(F)} e(ug)du.$$
In general, the integral above is not absolutely convergent. But it is absolutely convergent when $Re(\lambda)$ lies inside a positive cone, and it is $G(F)$-invariant. By restricting to $K$, we can view $J_{P'|P}(\tau_{\lambda})$ as a homomorphism from $I_{P}^{G}(V_{\tau_K})$ to $I_{P'}^{G}(V_{\tau_K})$. In general, $J_{P'|P}(\tau_{\lambda})$ can be meromorphically continued to a function on $\Fa_{M}^{\ast}\otimes_{\BR}\BC $. Moreover, if we assume that $\tau$ is tempered, we have the following proposition which is due to Harish-Chandra.

\begin{prop}\label{convergent}
With the notations above, assume that $\tau$ is tempered, then the intertwining operator $J_{P'|P}$ is absolutely convergent for all $\lambda \in \Fa_{M}^{\ast}\otimes_{\BR} \BC$ with $<Re(\lambda),\check{\alpha}>\; >0$ for every $\alpha\in \Sigma(P)\cap \Sigma(\bar{P}')$. Here $\Sigma(P)$ is the subsets of the roots of $A_M$ that are positive with respect to $P$.
\end{prop}

We will use this proposition in later sections to show some generalized Jacquet integrals are absolutely convergent.

\subsection{The open orbit method}
In this section we will give a brief overview of the open orbit method. The purpose of this method is to study the distinction of induced representations, it is an application of the geometric lemma due to Bernstein and Zelevinsky. Let $G$ be a connected reductive group defined over $F$, and $H\subset G$ is a closed subgroup such that $X=H\backslash G$ is a spherical variety of $G$ (i.e. the Borel subgroup has an open orbit). Let $P=MU$ be a parabolic subgroup of $G$ and $(\tau,V_{\tau})$ be an irreducible admissible representation of $M(F)$. We want to study the Hom space $Hom_{H(F)}(I_{P}^{G}(\tau),\chi)$ where $\chi$ is some character of $H(F)$. We say $(\pi,V_{\pi})=(I_{P}^{G}(\tau), I_{P}^{G}(V_{\tau}))$ is $(H,\chi)$-distinguished (or just H-distinguished if $\chi$ is trivial) if the Hom space is nonzero. For simplicity, we assume that $\chi$ is trivial.
\\
\\
\textbf{The geometric lemma} (Bernstein-Zelevinsky, \cite{BZ77}) There is an ordering $\{P(F)y_i H(F) \}_{i=1}^{N}$ on the double coset $H(F)\backslash G(F)/P(F)$ such that
$$Y_i=\cup_{j=1}^{i} P(F)y_i H(F)$$
is open in $G(F)$ for any $1\leq i\leq N$.

With the filtration above, for $1\leq i\leq N$, define
$$V_i=\{f\in I_{P}^{G}(V_{\tau})| supp(f)\subset Y_i\}.$$
Then we have $V_1\subset V_2\subset \cdots \subset V_N=V_{\pi}$ and $V_i$ is $H(F)$-invariant for all $i$. In particular, this implies that if $I_{P}^{G}(\tau)$ is H-distinguished, there exists $i$ such that $Hom_{H(F)}(V_i/V_{i-1},\chi)\neq 0$ (here $V_0=\{0\}$). Moreover, for any $1\leq i\leq N$, it is easy to see that the map
$$f\in V_i\mapsto \phi_f(h):=f(y_i h)$$
is an isomorphism between $V_i/V_{i-1}$ and $ind_{H_i}^{H} (\delta_{P}^{1/2} \tau^{y_i}|_{H_i})$ ($ind_{H_i}^{H}$ is the compact induction). Here $H_i=H(F)\cap y_{i}^{-1}P(F)y_i=y_{i}^{-1} P_i y_i$ with $P_i=P(F)\cap y_{i}H(F)y_{i}^{-1}$. By applying the reciprocity law, we have a necessary condition for $I_{P}^{G}(\tau)$ to be H-distinguished.

\begin{prop}
If $I_{P}^{G}(\tau)$ is H-distinguished, there exists $i$ such that $\tau$ is $(P_i,\delta_{P_i}\delta_{P}^{1/2})$-distinguished. Here we view $\tau$ as a representation of $P(F)$ by making it trivial on $U(F)$.
\end{prop}

What we are interested is the opposite direction of the proposition above. In other words, we want to have some sufficient conditions for $I_{P}^{G}(\tau)$ to be H-distinguished in terms of $V_i/V_{i-1}$. This is known as the open orbit method and the closed orbit method. For our purpose, we only consider the open orbit method.

Assume that $\tau$ is $(P_1,\delta_{P_1}\delta_{P}^{1/2})$-distinguished, we want to show that $\pi$ is $H$-distinguished. For simplicity, assume that $H(F)P(F)$ is open in $G(F)$ and $y_1=1$. Choose a nonzero element $l_0$ in the Hom space for $\tau$, it gives a nonzero element $l$ in $Hom_{H(F)}(V_1,1)$ by integrating $l_0$ over $H_1(F)\backslash H(F)$. Then we would like to extend this integral to $V_{\pi}$, which will gives us a nonzero element in $Hom_H(F)(V_{\pi},1)$. However, the integral will not be absolutely convergent in general, one need to show that it have holomorphic continuation. In our case, the integral over $H(F)/H_1(F)$ will be some generalized Jacquet integral. In Section 5 and 6, we will use Proposition \ref{convergent} to show that the integral is absolutely convergent for some $\pi$ with positive exponents, this will prove Theorem \ref{main} and Theorem \ref{F=R}. Then in section 7, we will talk about how to remove the restrains on the exponents by applying R. Gomez's result on the holomorphic continuation of generalized Jacquet integral.

\section{The relation between Conjecture \ref{jiang} and Conjecture \ref{epsilon}}
The goal of this section is to prove the following proposition.

\begin{prop}\label{epsilon 1}
If $F=\BC$, then Conjecture \ref{jiang} will implies Conjecture \ref{epsilon}.
\end{prop}

\begin{proof}
Since $F=\BC$, $\pi_D$ is always 0. Hence Conjecture \ref{jiang} tells us that the multiplicity $m(\pi)$ is always 1. Therefore in order to prove Conjecture \ref{epsilon}, it is enough to show that the epsilon factor $\epsilon(1/2,\pi,\wedge^3)$ equals to 1 for any irreducible generic representations $\pi$ of $\GL_6(F)$ with trivial central character.

By the Langlands classification, we can find a generic representation $\sigma=\sigma_1\times \sigma_2$ of $GL_5(F)\times \GL_1(F)$ such that $\pi$ is the parabolic induction of $\sigma$. Let $\phi$ be the Langlands parameter of $\pi$ and $\phi_i$ be the Langlands parameter of $\sigma_i$ for $i=1,2$, we have $\phi=\phi_1\oplus \phi_2$. This implies
\begin{eqnarray*}
\wedge^3(\phi)&=&\wedge^3(\phi_1\oplus \phi_2)=\wedge^3(\phi_1)\oplus (\wedge^2(\phi_1)\otimes \phi_2).
\end{eqnarray*}
Since the central character of $\pi$ is trivial, $\det(\phi)=\det(\phi_1)\otimes \det(\phi_2)=1$. Therefore $(\wedge^3(\phi_1))^{\vee}=\wedge^2(\phi_1)\otimes \det(\phi_1)^{-1}=\wedge^2(\phi_1)\otimes \det(\phi_2)=\wedge^2(\phi_1)\otimes \phi_2$, hence
$$\epsilon(1/2,\pi,\wedge^3)=\det(\wedge^3(\phi_1))(-1)=(\det(\phi_1))^6(-1)=1.$$
This finishes the proof of the proposition.
\end{proof}

\section{The tempered case}
In this section, we prove our main theorem for the tempered case, the method is very similar to the case $F=\BR$ we proved in \cite{Wan16}. Let $\pi$ be a tempered representation of $G=\GL_6(F)$ with central character $\chi^2$, our goal is to show that $m(\pi)=1$. Since we already know that $m(\pi)\leq 1$, it is enough to show that
\begin{equation}\label{1.1}
m(\pi)\neq 0.
\end{equation}

For all $T\in End(\pi)^{\infty}$, define
$$\CL_{\pi}(T)=\int_{\zh}^{\ast} Trace(\pi(h^{-1})T)\xi(h)\omega(h)dh.$$
Here $\int_{\zh}^{\ast}$ is the normalized integral defined in Proposition 5.1 of \cite{Wan16}. Note that the arguments in the loc. cit. is for the case when $F=\BR$, but they also work for $F=\BC$. By Lemma 5.2 of the loc. cit., for any $h,h'\in H(F)$, we have
\begin{equation}\label{1.2}
\CL_{\pi}(\pi(h)T\pi(h'))=\xi(hh')\omega(hh')\CL_{\pi}(T).
\end{equation}
For $e,e'\in \pi$, define $T_{e,e'}\in End(\pi)^{\infty}$ to be
$e_0\in \pi\mapsto (e_0,e')e.$
Set $\CL_{\pi}(e,e')=\CL_{\pi}(T_{e,e'})$, then we have
$$\CL_{\pi}(e,e')=\int_{\zh}^{\ast} (e,\pi(h)e')\omega(h)\xi(h) dh.$$
If we fix $e'$, by \eqref{1.2}, the map $e\in \pi\rightarrow \CL_{\pi}(e,e')$ belongs to $Hom_H(\pi,\omega\otimes \xi)$. Since $Span \{T_{e,e'}\mid e,e'\in \pi\}$ is dense in $End(\pi)^{\infty}$, we have
$$\CL_{\pi}\neq 0\Rightarrow m(\pi)\neq 0.$$
Hence in order to show the multiplicity $m(\pi)$ is nonzero, it is enough to show that the operator $\CL_{\pi}$ is nonzero.

Since we are in the complex case, only $\GL_1(F)$ have discrete series, hence $\pi$ is a principal series. Let $R=M_RU_R$ be a good minimal parabolic subgroup of $G$ in the sense that $RH$ is Zariski open in $G$. The existence of such $R$ is proved in Proposition 4.2 of \cite{Wan16}. It is also proved in the same proposition that for all such $R$, we have $H_R:=H\cap R=Z_G$. Hence the reduced model associated to $R$ is just $(M_R, Z_G)$. Since $\pi$ is a principal series, there is an unitary character $\tau$ of $M_R$ such that $\pi=I_{R}^{G}(\tau)$. For $T_0\in End(\tau)^{\infty}$, define
$$\CL_{\tau}(T_0)=Trace(T_0).$$
By Proposition 5.9 of \cite{Wan16}, the nonvanishing property of $\CL_{\pi}$ is invariant under the parabolic induction, hence we have
$$\CL_{\pi}\neq 0\iff \CL_{\tau}\neq 0.$$
Here the arguments in the loc. cit. is for the case when $F=\BR$, but they also work for $F=\BC$. Since $\CL_{\tau}$ is obviously nonzero, we have $\CL_{\pi}\neq 0$. This proves $m(\pi)\neq 0$ and hence finishes the proof of Theorem \ref{main} for tempered representations.

\section{The proof of Theorem \ref{main}}
\subsection{The case when $\bar{P}\subset Q$}
In this section, we prove the first part of Theorem \ref{main}. In other word, we assume that $\bar{P}\subset Q$. Then there are four possibilities for $Q$: type $(6)$, type $(4,2)$, type $(2,4)$ or type $(2,2,2)$. The idea is to first reduce our problem to the reduced model $(L,H\cap Q)$ by the open orbit method, then reduce it to the tempered case which has been considered in the previous section.

\textbf{If $Q=G$ is of type $(6)$}, by twisting $\pi$ by some characters, we can assume that $\pi$ is tempered. Note that twisting by characters will not change the multiplicities. Then by applying the result in the last section, we know $m(\pi)\neq 0$ and this proves Theorem \ref{main}.

\textbf{If $Q$ is of type $(4,2)$}, then $L=\GL_4(F)\times \GL_2(F)$ and $H_Q=H\cap Q$ is of the form
$$H_Q=H_0 U_{0,Q}$$
where
$$U_{0,Q}(F)=\{ u=u(X):=\begin{pmatrix} 1 & X & 0 \\ 0 & 1 & 0 \\ 0 & 0 & 1 \end{pmatrix}| \; X\in M_2(F)\}.$$
The restriction of the character $\xi$ on $U_{0,Q}(F)$ is just $\xi(u(X))=\psi(\tr(X))$ and the character $\omega$ on $H_0$ is defined as usual. The model $(L,H_Q)$ is middle model introduced in \cite{Wan15}, it can be understood as the model between the Ginzburg-Rallis model and the trilinear $\GL_2$ model. By the definition of $Q$, $\pi$ is of the form $I_{Q}^{G}(\tau_1 |\;|^{t_1}\times \tau_2 |\;|^{t_2})$ where $\tau_1, \tau_2$ are tempered and $t_1<t_2$. Hence any element $f\in \pi$ is a smooth function $f:G(F)\rightarrow \tau=\tau_1 |\;|^{t_1}\times \tau_2 |\;|^{t_2}$ such that
\begin{equation}\label{2.3}
f(lug)=\delta_Q(l)^{1/2}\tau(l)f(g)
\end{equation}
for all $l\in L(F),u\in U_Q(F)$ and $g\in G(F)$. Here we use the letters $\pi,\sigma,\tau$ to denote both the representations and the underlying vector spaces.
Let $\bar{Q}=LU_{\bar{Q}}$ be the opposite parabolic subgroup of $Q$, then it is easy to see that $U_{\bar{Q}}\subset U$ and $U=U_{\bar{Q}}U_{0,Q}$. For any $f\in \pi$, define
\begin{equation}\label{2.1}
J_Q(f)=\int_{U_{\bar{Q}}(F)} f(u)\xi^{-1}(u)du.
\end{equation}
By Proposition \ref{convergent} together with the assumption that $t_1<t_2$, the integral above is absolutely convergent.

\begin{prop}\label{2.2}
\begin{enumerate}
\item For all $f\in \pi, u\in U_{\bar{Q}}(F)$ and $l\in H_Q(F)$, we have
\begin{equation}\label{2.4}
J_Q(\pi(u)f)=\xi(u)J(f)
\end{equation}
and
\begin{equation}\label{2.5}
J_Q(\pi(l)f)=\tau(l)J(f).
\end{equation}
\item The function
$$J_Q:\pi\rightarrow \tau,\; f\rightarrow J_Q(f)$$
is surjective.
\end{enumerate}
\end{prop}

\begin{proof}
Part (1) follows from \eqref{2.3} and changing variables in the integral \eqref{2.1}. For part (2), fix a function $\varphi\in C_{c}^{\infty}(U_{\bar{Q}}(F))$ such that $\int_{U_{\bar{Q}}(F)} \varphi(u) \psi^{-1}(u)du=1$. For any $v\in \tau$, since $Q(F)U_{\bar{Q}}(F)$ is open in $G(F)$, the function
$$f(g)=\begin{array}{cc}\left\{ \begin{array}{ccl} \delta_{Q}(l)^{1/2}\tau(l)\varphi(u)v & if & g=u'lu\; with\; l\in L(F),u\in U_{\bar{Q}}(F),u'\in U_Q(F) \\ 0 & else & \\ \end{array}\right. \end{array}$$
lies inside $\pi$. Then we have
$$J_Q(f)=\int_{U_{\bar{Q}}(F)} f(u)\psi^{-1}(u)du=\int_{U_{\bar{Q}}(F)} \varphi(u)\psi^{-1}(u) v du=v.$$
This proves (2).
\end{proof}

We consider the Hom space $Hom_{H_Q(F)}(\tau,(\omega\otimes \xi)|_{H_Q(F)})$ and let $m(\tau)$ be the dimension of such space. The following proposition tells the relation between $m(\pi)$ and $m(\tau)$.
\begin{prop}
$$m(\tau)\neq 0\Rightarrow m(\pi)\neq 0.$$
\end{prop}

\begin{proof}
If $m(\tau)\neq 0$, choose $0\neq l_0\in Hom_{H_Q(F)}(\tau,(\omega\otimes \xi)|_{H_Q(F)})$, define an operator $l$ on $\pi$ to be
$$l(f)=l_0(J_Q(f)).$$
Since $l_0\neq 0$ and $J_Q$ is surjective, we have $l\neq 0$. Hence we only need to show that $l\in Hom_{H(F)}(\pi,\omega\otimes \xi)$.

For $h\in H(F)$, we can write $h=h_1u_1$ with $h_1\in H_{Q}(F)$ and $u_1\in U_{\bar{Q}}(F)$. By \eqref{2.4} and \eqref{2.5}, we have
\begin{eqnarray*}
l(\pi(h)f)&=&l_0(J_Q(\pi(h_1u_1)f))=l_0(\tau(h_1)J_Q(\pi(u_1)f))\\
&=&\omega\otimes \xi(h_1) l_0(J_Q(\pi(u_1)f))=\omega\otimes \xi(h_1) l_0(\xi(u_1)J_Q(f))\\
&=&\omega\otimes \xi(h) l_0(J_Q(f))=\omega\otimes \xi(h) l(f).
\end{eqnarray*}
This implies $l\in Hom_{H(F)}(\pi,\omega\otimes \xi)$ and finishes the proof of the Proposition.
\end{proof}

By the proposition above, we only need to show that $m(\tau)\neq 0$. It is easy to see that the multiplicity $m(\tau)$ is invariant under the unramified twist, hence we may assume that $\tau$ is tempered (note that originally $\tau$ is of the form $\tau_1 |\;|^{t_1}\times \tau_2 |\;|^{t_2}$ with $\tau_1$ and $\tau_2$ being tempered). Then by applying the argument in the previous section to the middle model case, we can show that the multiplicity $m(\tau)$ is always nonzero for all tempered representations $\tau$. This proves Theorem \ref{main}.

\textbf{If $Q$ is of type $(2,4)$}, the argument is the same as the $(4,2)$ case, we will skip it here.

\textbf{If $Q$ is of type $(2,2,2)$}, the argument is still similar to the $(4,2)$ case: we first reduce to the trilinear $\GL_2$ model case by the open orbit method. Then after twisting by some characters we only need to consider the tempered case. Finally, by applying the argument in the previous section to the trilinear $\GL_2$ model case, we can show that the multiplicity is nonzero and this proves Theorem \ref{main}. We will skip the details here.

Now the proof of Theorem \ref{main}(1) is complete.

\subsection{The case when $Q\subset \bar{P}$}
In this section, we prove the second part of Theorem \ref{main}. Recall that in Section 1 we assume that $\pi=Ind_{B}^{G}(\Pi_{i=1}^{6} \chi_i)$ where $B$ is the lower Borel subgroup, $\chi_i=\sigma_i |\;|^{s_i}$, $\sigma_i$ are unitary characters, and $s_i$ are real numbers with $s_1\leq s_2\leq \cdots\leq s_6$. By the assumption $Q\subset \bar{P}$, we have $s_2<s_3$ and $s_4<s_5$. Also as in Section 1, we write $\pi=I_{\bar{P}}^{G}(\pi_0)$ with $\pi_0=\pi_1\times \pi_2\times \pi_3$ and $\pi_i$ be the parabolic induction of $\chi_{2i-1}\times \chi_{2i}$. Then $\pi$ is consisting of smooth functions $f\rightarrow \pi_0$ such that
\begin{equation}\label{3.1}
f(mug)=\delta_{\bar{P}}(m)^{1/2}\pi_0(m)f(g)
\end{equation}
for all $m\in M(F),u\in \bar{U}(F)$ and $g\in G(F)$. We still want to apply the open orbit method. For $f\in \pi$, define
\begin{equation}\label{5.3}
J(f)=\int_{U(F)} f(ug)\xi^{-1}(u) du.
\end{equation}
By Proposition \ref{convergent} together with the assumption on the exponents $s_i$, the integral above is absolutely convergent. Similarly as in the previous section, we can show that
\begin{equation}\label{5.2}
m(\pi_0)\neq 0\Rightarrow m(\pi)\neq 0.
\end{equation}
Here $m(\pi_0)$ is the multiplicity for the trilinear $\GL_2$ model. In fact, for $0\neq l_0\in Hom_{H_{0}(F)}(\pi_0,\omega)$. By a similar argument as in Proposition 5.2, we know that
$$l(f):=l_0(J(f))$$
is a nonzero element in $Hom_{H(F)}(\pi,\omega\otimes \xi)$. This proves \eqref{5.2}. Now by our assumption on $\pi_0$ together with the work by Loke for the trilinear $\GL_2$ model in \cite{L01}, we know $m(\pi_0)\neq 0$, hence we have $m(\pi)\neq 0$. This finishes the proof of Theorem \ref{main}.

\begin{rmk}
The assumption $Q\subset \bar{P}$ is only used to make the generalized Jacquet integral $J(f)$ to be absolutely convergent. Hence in general, if one can prove the holomorphic continuation of the generalized Jacquet integral $J(f)$, then the assumption $Q\subset \bar{P}$ in Theorem \ref{main}(2) can be removed. This will be discussed in Section 7.
\end{rmk}

\section{The proof of Theorem \ref{F=R}}
In this section by applying the open orbit method to the case when $F=\BR$, we prove Theorem \ref{F=R}. Let $\pi$ be an irreducible generic representation of $G(F)$ with central character $\chi^2$. With the same notation as in Section 1, there is a parabolic subgroup $Q=LU_Q$ containing the lower Borel subgroup and an essential tempered representation $\tau=\Pi_{i=1}^{k} \tau_i |\;|^{s_i}$ of $L(F)$ with $\tau_i$ tempered, $s_i\in \BR$ and $s_1<s_2<\cdots<s_k$ such that $\pi=I_{Q}^{G}(\tau)$.

\subsection{The case when $\pi_D=0$}
In this section we assume that $\pi_D=0$. Then by our assumptions in Theorem \ref{F=R}, $Q$ is nice. If $Q\subset \bar{P}$, let $\pi_0=I_{Q\cap M}^{M}(\tau)$, it is a generic representation of $M(F)$ and we have $\pi=I_{\bar{P}}^{G}(\pi_0)$. By the same argument as in Section 5.2, we can show that
\begin{equation}\label{5.1}
m(\pi_0)\neq 0\Rightarrow m(\pi)\neq 0
\end{equation}
where $m(\pi_0)$ is the multiplicity of the trilinear $\GL_2$ model. Since $\pi_D=0$, the Jacquet-Langlands correspondence of $\pi_0$ from $M(F)=(\GL_2(F))^3$ to $(\GL_1(D))^3$ is zero. By applying the result for the trilinear $\GL_2$ model in \cite{P90} and \cite{L01}, we have $m(\pi_0)=1$. Combining with \eqref{5.1}, we know $m(\pi)\neq 0$. Hence $m(\pi)=1$ since we already know $m(\pi)\leq 1$. Therefore
$$m(\pi)+m(\pi_D)=m(\pi)=1.$$
This proves Conjecture \ref{jiang}. For Conjecture \ref{epsilon}, we only need to show that when $\pi_D=0$, the epsilon factor $\epsilon(1/2,\pi,\wedge^3)$ is always 1. Since $\pi_{D}=0$, by the local Jacquet-Langlands correspondence in \cite{DKV84}, $\pi_0$ is not essential discrete series (i.e. the discrete series twisted by characters), hence at least one of the $\pi_i\; (i=1,2,3)$ is a principal series. Therefore we can find a generic representation $\sigma=\sigma_1\times \sigma_2$ of $\GL_5(F)\times \GL_1(F)$ such that $\pi$ is the parabolic induction of $\sigma$. Then by the same argument as in Section 3, we can show that
$$\epsilon(1/2,\pi,\wedge^3)=1.$$
This finishes the proof of Conjecture \ref{epsilon}.

If $Q\subset \bar{P}$, there are only four possibilities for $Q$: type $(6),(4,2),(2,4)$ and $(2,2,2)$. If $Q$ is type $(6)$, by twisting $\pi$ by some characters we can assume that $\pi$ is tempered, then both Conjecture \ref{jiang} and Conjecture \ref{epsilon} are proved in \cite{Wan16}. If $Q$ is type $(4,2)$ or $(2,4)$, by the same argument as in Section 5.1, we can reduce to the middle model case by the open orbit method. Then by twisting some characters, we only need to consider the tempered case which has already been proved in \cite{Wan16}. If $Q$ is type $(2,2,2)$, the argument is similar except replacing the middle model by the trilinear $\GL_2$ model.

Now the proof of Theorem \ref{F=R}(1) is complete.

\subsection{The case when $\pi_D\neq 0$}
In this section we assume that $\pi_D\neq 0$, as a result, $\pi=I_{\bar{P}}^{G}(\pi_0)$ is the parabolic induction of some essential discrete series $\pi_0=\pi_1 |\;|^{s_1}\times \pi_2 |\;|^{s_2} \times \pi_3 |\;|^{s_3}$ of $M(F)$ where $\pi_i$ are discrete series of $\GL_2(F)$ and $s_i$ are real numbers. As usual, we assume that $s_1\leq s_2\leq s_3$. On the mean time, $\pi_D$ is of the form $I_{\bar{P}_D}^{G_D}(\pi_{0,D})$ where $\pi_{0,D}=\pi_{1,D} |\;|^{s_1}\times \pi_{2,D} |\;|^{s_2} \times \pi_{3,D} |\;|^{s_3}$ is the Jacquet-Langlands correspondence of $\pi_0$ from $M(F)$ to $M_D(F)$. Let $m(\pi_0)$ (resp. $m(\pi_{0,D})$) be the multiplicity of the trilinear $\GL_2(F)$ (resp. $\GL_1(D)$) model.

\begin{prop}
With the notations above, in order to prove Theorem \ref{F=R}(2), it is enough to show that
\begin{equation}\label{6.1}
m(\pi_0)\neq 0\Rightarrow m(\pi)\neq 0;\; m(\pi_{0,D})\neq 0\Rightarrow m(\pi_D)\neq 0.
\end{equation}
\end{prop}

\begin{proof}
By Prasad's result for the trilinear $\GL_2$ model, we have
\begin{equation}\label{6.2}
m(\pi_0)+m(\pi_{0,D})=1.
\end{equation}
Moreover, if we assume that the central character of $\pi_0$ is trivial on $H_0(F)$, we have
\begin{equation}\label{6.3}
m(\pi_0)=1 \iff \epsilon(1/2,\pi_0)=1; \; m(\pi)=0 \iff \epsilon(1/2,\pi_0)=-1.
\end{equation}

Combining \eqref{6.1} and \eqref{6.2}, we have $m(\pi)+m(\pi_D)\geq 1$, this proves the first part of Theorem \ref{F=R}(2). For the second part, assume that the central character of $\pi$ is trivial, in Section 6.2 of \cite{Wan16}, we proved that
\begin{equation}\label{6.4}
\epsilon(1/2,\pi,\wedge^3)=\epsilon(1/2,\pi_0).
\end{equation}
Now if $\epsilon(1/2,\pi,\wedge^3)=1$, by \eqref{6.4}, we know $\epsilon(1/2,\pi_0)=1$. Combining with \eqref{6.3}, we have $m(\pi_0)=1$, therefore $m(\pi)=1$ by \eqref{6.1}. On the other hand, if $m(\pi)=0$, by \eqref{6.1}, we have $m(\pi_0)=0$. Combining with \eqref{6.3}, we have $\epsilon(1/2,\pi_0)=-1$, therefore $\epsilon(1/2,\pi,\wedge^3)=-1$ by \eqref{6.4}. This finishes the proof of Theorem \ref{F=R}(2).
\end{proof}

By the proposition above, it is enough to show \eqref{6.1}. If $s_1=s_2=s_3$, by twisting $\pi$ by some characters, we may assume that $\pi$ is tempered (note that the multiplicities for both the Ginzburg-Rallis model the the trilinear $\GL_2$ model are invariant under twisting by characters). Then the relation \eqref{6.1} has already been proved in Corollary 5.13 of \cite{Wan16}. In fact, in this case, we even know $m(\pi)=m(\pi_0)$ and $m(\pi_D)=m(\pi_{0,D})$.

If $s_1<s_2=s_3$, let $\pi_{2,3}$ be the parabolic induction of $\pi_2\times \pi_3$, it is a tempered representation of $\GL_4(F)$. We also know that $\pi$ will be the parabolic induction of $\pi'=\pi_1 |\;|^{s_1}\times \pi_{2,3} |\;|^{s_2}$. Let $m(\pi')$ be the multiplicity for the middle model. By applying the open orbit method as in Section 5.1, we have
$$m(\pi')\neq 0\Rightarrow m(\pi)\neq 0.$$
Hence in order to prove $m(\pi_0)\neq 0\Rightarrow m(\pi)\neq 0$, it is enough to show that $m(\pi_0)\neq 0\Rightarrow m(\pi')\neq 0$. Again by twisting $\pi'$ by some characters, we may assume that $\pi'$ is tempered. Then by Corollary 5.13 of \cite{Wan16}, we have $m(\pi_0)=m(\pi')$ which implies $m(\pi_0)\neq 0\Rightarrow m(\pi)\neq 0$. The proof of the quaternion version is similar. This proves \ref{6.1}.

If $s_1=s_2<s_3$, the argument is the same as the case above, we will skip it here.

If $s_1<s_2<s_3$, \ref{6.1} follows directly from the open orbit method as in Section 5.1.

Now the proof of Theorem \ref{F=R}(2) is complete.

\section{Holomorphic continuation of the generalized Jacquet integral}
In previous sections, we have already seen that the extra conditions of $Q$ in Theorem \ref{main}(2) and Theorem \ref{F=R}(1) can be removed if the generalized Jacquet integral $J(f)$ defined in \ref{5.3} has holomorphic continuation. In this section, we are going to remove the condition on $Q$ based on the following hypothesis.

\textbf{Hypothesis:} The generalized Jacquet integrals have holomorphic continuation for all parabolic subgroups whose unipotent radical is abelian.

The Hypothesis has been proved by Gomez and Wallach in \cite{GW12} for the case when the stabilizer of the unipotent character is compact, and proved by Gomez in \cite{G} for the general case. The second paper is still in preparation, this is why we write it as a hypothesis.

Let $F=\BR$ or $\BC$, $\pi$ be a generic representation of $\GL_6(F)$ of the form $\pi=I_{\bar{P}}^{G}(\pi_0)$ for some generic representation $\pi_0$ of $M(F)=(\GL_2(F))^3$. By the discussion in Section 5.2 and 6.1, we know that in order to prove Theorem \ref{main}(2) and Theorem \ref{F=R}(1) for $\pi$, it is enough to prove
\begin{equation}\label{7.1}
m(\pi_0)\neq 0\Rightarrow m(\pi)\neq 0.
\end{equation}
where $m(\pi_0)$ is the multiplicity for the trilinear $\GL_2$ model.

Let $Q_{4,2}=L_{4,2}U_{4,2}$ be the parabolic subgroup of $\GL_6(F)$ containing $\bar{P}$ of type $(4,2)$, and let $\pi_1=I_{\bar{P}\cap L_{4,2}}^{L_{4,2}}(\pi_0)$. Then in order to prove \eqref{7.1}, it is enough to show that
\begin{equation}\label{7.2}
m(\pi_0)\neq 0\Rightarrow m(\pi_1)\neq 0,\;m(\pi_1)\neq 0\Rightarrow m(\pi)\neq 0
\end{equation}
where $m(\pi_1)$ is the multiplicity for the middle model defined in Section 5.1. Note that the unipotent radicals of $Q_{4,2}$ and $\bar{P}\cap L_{4,2}$ are all abelian. Therefore by the hypothesis, the generalized Jacquet integrals associated to $Q_{4,2}$ and $\bar{P}\cap L_{4,2}$ have holomorphic continuation. This allows us to apply the open orbit method as in Section 5 and 6, which give the relations in \eqref{7.2}. This proves \eqref{7.1}, and finishes the proof of Theorem \ref{main}(2) and Theorem \ref{F=R}(1) without the assumptions on $Q$.


\end{document}